\documentclass[12pt, english,reqno]{amsart}
\usepackage[margin=1.25in]{geometry}

\usepackage{amsmath,amsthm,amssymb,mathtools}
\usepackage[T1]{fontenc}
\usepackage[utf8]{inputenc} 
\usepackage{microtype}
\usepackage{enumitem}
\usepackage{xcolor}
\usepackage{bbm}
\usepackage{hyperref}
\usepackage[nameinlink,noabbrev]{cleveref}

\hypersetup{
  colorlinks=true,
  linkcolor=blue,
  citecolor=blue,
  urlcolor=blue
}

\usepackage[
  backend=biber,
  style=alphabetic,
  maxnames=5,       
  minnames=5,       
  maxbibnames=5,    
  minbibnames=5,     
  maxalphanames=5,  
  minalphanames=5
]{biblatex}
\numberwithin{equation}{section}

\theoremstyle{plain}
\newtheorem{theorem}{Theorem}[section]
\newtheorem{lemma}[theorem]{Lemma}

\newtheorem{corollary}[theorem]{Corollary}

\theoremstyle{definition}
\newtheorem{definition}[theorem]{Definition}

\theoremstyle{remark}
\newtheorem{remark}[theorem]{Remark}

\newcommand{\R}{\mathbb{R}}
\newcommand{\N}{\mathbb{N}}

\newcommand{\Sph}{\mathbb{S}}

\newcommand{\ep}{\varepsilon}

\author{Xavier Fern\'andez-Real}
\address{EPFL SB, Station 8, 1015 Lausanne, Switzerland}
\email{xavier.fernandez-real@epfl.ch}

\author{Enric Florit-Simon}
\address{ETH Zurich, R\"amistrasse 101, 8092 Zurich,
Switzerland}
\email{enric.florit@math.ethz.ch}

\author{Joaquim Serra}
\address{ETH Zurich, R\"amistrasse 101, 8092 Zurich,
Switzerland}
\email{joaquim.serra@math.ethz.ch}


\title{Improvement of flatness in annuli}

\dedicatory{Dedicated to the memory of Louis Nirenberg, who taught us that\\ revisiting classical results through new proofs is often a path to progress.}

\thanks{X. F. was supported by the Swiss National Science Foundation (SNF grant PZ00P2\_208930), by
the Swiss State Secretariat for Education, Research and Innovation (SERI) under contract number
MB22.00034, and by the AEI project PID2021-125021NA-I00 (Spain). E. F. and J. S. were supported by the European Research Council under the Grant Agreement No 948029.}

\date{} 

\begin{document}

\maketitle

\begin{abstract}
We present a short and flexible improvement-of-flatness argument adapted to the setting of exterior domains, where one is naturally led to work with annuli instead of balls.

As a model application in the classical setting of minimal surfaces, we give an alternative proof of the end-structure and asymptotics for finite Morse index minimal hypersurfaces with Euclidean area growth in low dimensions.

The method is largely PDE-based and general in its application. Suitable variants have been employed in Bernoulli and Allen--Cahn settings.
\end{abstract}

\medskip 

\section{Introduction}

Improvement-of-flatness type arguments are a basic tool in the  regularity theory for nonlinear geometric PDE (minimal surfaces,
free boundaries, phase transitions etc.).  In a typical formulation, sufficient flatness of a solution---or interface---in $B_1$, measured by proximity to a model profile (often a lower-dimensional plane), implies improved flatness after rescaling to a smaller ball $B_\rho$, possibly relative to a slightly adjusted profile. This paradigm goes back to De~Giorgi's work \cite{degiorgi1961frontiere} on area-minimizing hypersurfaces and has many extensions in geometric measure theory and PDEs: for example, in minimal surfaces \cite{allard1972varifold,taylor1976soap, Taylor1973Mod3, White1979Mod4, Sim93, SS81};
free boundary problems
\cite{altcaffarelli1981onephase,caffarelli1987harnackI,desilva2011rhs,desilvasavin2019degenerate,dephilippisspolaorvelichkov2021twophase,
altcaffarellifriedman1984twophase, weiss1999homogeneity, Weiss1999FreeBoundary, fernandezreal2020regularity}; and semilinear PDEs and systems \cite{CaffarelliLin2008, Savin2009FlatLevelSets,Wang17, DePhilippisHalavatiPigati2024DecayExcess,AudritoSerra2021} 

When one studies \emph{exterior} problems---for example, the ends of embedded minimal hypersurfaces,
or exterior free boundary configurations---the natural scale-invariant domains are annuli rather
than balls.  In this setting one is led to track flatness on dyadic annuli
$B_{2r}\setminus B_{r/2}$ as $r\to\infty$ (or $r\to 0$).  In the present work, we record a short
and simple annular improvement-of-flatness scheme for such a setting.
 
The main difference compared with standard improvement-of-flatness results in balls is the following. Consider a solution on an annulus \(B_{r_2}\setminus B_{r_1}\) with a large scale separation
\(r_2/r_1\gg 1\), and assume the solution is sufficiently flat at every intermediate scale between
\(r_1\) and \(r_2\).
In this annular setting, flatness improves geometrically as one moves inward from the outer scale
\(r_2\), but only down to a \emph{mesoscale}
\begin{equation}
\label{eq:rast}
r_* \;=\; r_1^{\beta}\, r_2^{\,1-\beta},
\qquad \beta\in(0,1),
\end{equation}
where \(\beta\) depends on the problem.
Below \(r_*\), the trend reverses: flatness \emph{deteriorates}, again at a geometric rate, as one
continues inward from \(r_*\) down to \(r_1\).
This loss of flatness reflects the existence, in annular domains, of ``modes'' (solutions of the
linearized problem) associated with negative homogeneities.

Of course, from a coarse or “low-resolution” perspective, this idea is not new, and implicit variants have appeared in several contexts before; see also the symmetric (log-)epiperimetric approach of \cite{EdelenSpolaorVelichkov2024}. However, it leads rather quickly to significant progress in quite well-known problems---such as the classification of finite-index solutions to the Allen–Cahn and Bernoulli problems (see below)---suggesting that this mechanism is not yet part of the ``standard analytical toolbox''. Given the simplicity and flexibility of the method, we believe it is worthwhile to record it explicitly here, both to prevent it from being overlooked and to facilitate its adaptation to other settings.

For concreteness, we apply the argument to the classical setting of minimal surfaces, where we give an alternative
proof of the known end-structure and asymptotics for finite Morse index minimal hypersurfaces with
Euclidean area growth in dimensions $3\le n\le 7$.   

It should be clear from the proof that the approach is flexible, and that essentially the same line of reasoning applies in a wide range of settings in which an improvement-of-flatness estimate is available, in the spirit of De Giorgi, Allard, Savin, and others.

As a matter of fact, the annular iteration scheme presented here can be viewed as the ``baby version'' (or conceptual seed) of one of the central arguments  introduced in \cite{CFFS25} to study stable solutions of the  Bernoulli free boundary problem in three dimensions. In that setting, a similar scheme appears in a substantially more elaborate form: the natural linearization domains are no longer annuli, but rather balls from which a (possibly large) number of smaller balls have been removed, the so-called ``neck balls''. The annular case here presented would correspond to the situation in which all ``neck balls'' happen to be clustered inside a fixed ball around the origin. Within \cite{CFFS25}, the simple underlying mechanism discussed below is embedded into a considerably more involved framework, which obscures its essential structure.

Our goal is to isolate this mechanism and present it in a transparent geometric setting, so that it can be readily recognised and adapted elsewhere.  There is, moreover, a number of recent works where these ideas are applied:
\begin{itemize}
\item As mentioned, clear  analogies can be drawn between the annular improvement of flatness and one for stable solutions of the Bernoulli problem in \cite{CFFS25}. The same occurs in \cite{FS25} in the setting of stable solutions to the Allen-Cahn equation with energy density bounds.
    \item The particular version presented here---which is for minimal hypersurfaces---finds its root in the study of global solutions to the Bernoulli problem with finite index instead. We refer the reader to the article \cite{FFS25}, which shows that global solutions to the Bernoulli problem with finite index in three dimensions are axially symmetric.
    \item A further variant is also an essential starting point in \cite{Flo25}, which establishes parallels to a well-known conjecture of De Giorgi for solutions to the Allen--Cahn equation with finite index.
    \item While the cases listed above pertain to recent works, the flexibility of the method makes it applicable far more broadly: it can be used to study “exterior geometry” in essentially any problem in which some form of improvement of flatness is available.
    \end{itemize}

\section{Main results}
We recall that a minimal hypersurface $\Sigma\subset D$ is a critical point of the area functional in $D$, and it is geometrically characterised by having mean curvature ${\rm H}_\Sigma\equiv 0$. Whenever it can be written (locally) as a graph of some function $f$, this graph then satisfies the minimal graph equation
$$
{\rm div}\left(\frac{\nabla f}{\sqrt{1+|\nabla f|^2}}\right)= 0.
$$
We will always work with smooth minimal hypersurfaces satisfying suitable completeness assumptions. To be precise, unless otherwise specified, the statement ``$\Sigma$ is a minimal hypersurface in $D$'' will mean:
\begin{itemize}
    \item $D\subset \R^n$ is an open domain.
    \item $\Sigma\cap D$ is a smooth, embedded, $(n-1)$-dimensional submanifold, without boundary in the manifold sense, and satisfying the completeness condition $\overline{(\Sigma\cap D)}\setminus (\Sigma\cap D)\subset\partial D$.
    \item There holds ${\rm H}_{\Sigma\cap D}\equiv 0$.
\end{itemize}

To state our main result, it is convenient to define:
 \begin{definition}[Annular heights]\label{defi:heightminimal} Let $n\ge 3$. Let $\Sigma$ be a minimal hypersurface in $D\subset\R^n$. We define the \emph{shifted annular height} in the direction $e\in \mathbb S^{n-1}$, with shift $b\in \R$, and at scale $r>0$ with $B_{2r}\setminus \overline{B_{r/2}}\subset D$, as:
\[
    H_b(\Sigma,e,r) := \inf\big\{  h>0 \  : |e\cdot x-b|\leq h\quad\mbox{on}\quad\Sigma\cap (B_{2r}\setminus \overline{B_{r/2}}) \big\}.
\]
The \emph{centered annular height} is the particular case $b=0$, $H_0(\Sigma, e, r)$.

We also define 
\[H(\Sigma,r) : = \inf_{e\in \mathbb{S}^{n-1},\, b \in \R} H_b(\Sigma,e, r)\qquad \mbox{and}\qquad H_0(\Sigma,r) : = \inf_{e\in \mathbb{S}^{n-1}} H_0(\Sigma,e, r).\]
\end{definition}

Notice that we have the scaling
\begin{equation}
    \label{eq:scalingminimal}
    H_b(\Sigma, e, r) = \rho H_{b/\rho} (\Sigma_\rho, e, r/\rho),\qquad\text{for}\quad \rho  >0 \quad \text{and} \quad \Sigma_\rho(x) := \frac{1}{\rho}\Sigma. 
\end{equation}

Our improvement-of-flatness result in annuli is:
\begin{theorem}\label{thm:improvminimalH}
 Let $n \ge 3$.  Let $\Sigma$ be a minimal hypersurface in $B_R\setminus \overline {B_1}$ for some $R\geq 4$. Let $C_0>0$ and $\alpha\in(0,1)$. There are $C=C(n,\alpha,C_0)$ and $\eta_0=\eta_0(n,\alpha,C_0)>0$ such that the following holds:

Let $\eta\leq\eta_0$, and suppose that for every $r\in [2,R/2]$ we have that $\Sigma\cap (B_{2r}\setminus \overline{B_{r/2}})$ is connected, $|{\rm II}_{\Sigma\cap (B_{2r}\setminus \overline{B_{r/2}})}| \le C_0/r$, and
\begin{equation}\label{flatnesshypothesisthmminimal}
    H_0(\Sigma,r) \le \eta r.
\end{equation}
Then,
\begin{equation}\label{goalminimal}
H(\Sigma,r) \le C\eta r\left[ r^{3-n-\alpha} + (r/R)^\alpha \right] \qquad\text{for all}\quad r\in (2, R/2). 
\end{equation}
\end{theorem}
\begin{remark}\label{rem23}
    The two contributions in \eqref{goalminimal} correspond to the two competing mechanisms described in the introduction: the term
$(r/R)^{\alpha}$ propagates the improvement coming from the outer scale $R$, whereas $r^{3-n-\alpha}$ is the
contribution of the ``bad'' (inner) mode that eventually prevents a monotone inward improvement in general.
The associated mesoscale $r_*$ in \eqref{eq:rast} is obtained by balancing the two terms,
\[
r_\ast^{\,3-n-\alpha}=(r_\ast/R)^{\alpha}
\qquad\Longleftrightarrow \qquad
r_\ast = R^{\frac{\alpha}{\,n+2\alpha-3\,}} .
\]
By scaling back to a general annulus $B_{r_2}\setminus B_{r_1}$, this takes the form 
\[
r_\ast = r_1^{\beta}\,r_2^{1-\beta},
\qquad
\beta=\frac{n+\alpha-3}{n+2\alpha-3}\in(0,1),
\]
(cf. \eqref{eq:rast}). 
\end{remark}

\vspace{2mm}

To state our next result, we introduce some further notation.
Throughout the paper, $B'_r(y)$ denotes the ball in $\R^{n-1}$ of radius $r$ and center $y\in\R^{n-1}$, and $B_r':=B_r'(0)$. Given a set $\Omega \subset \mathbb{R}^{n-1}$ and a function $g : \Omega \to \mathbb{R}$, we denote by ${\rm graph}\, g$ the set of points in $\Omega \times \mathbb{R}$ satisfying $x_n = g(x_1,...,x_{n-1})$. \\

A model application of Theorem~\ref{thm:improvminimalH}, for which a new approach will be outlined at the end of the article, is the following classical result which follows from the arguments in \cite{Tysk89,Schoen83} (see also \cite{Anderson84, FC85, SS81}):
\begin{theorem}\label{thm:TyskIndex}
    Let $3\leq n \leq 7$ and $\beta\in(0,1)$. Let $\Sigma\subset\R^n$ be a complete, embedded minimal hypersurface with finite Morse index, satisfying ${\rm Area}(\Sigma\cap B_R)\leq CR^{n-1}$ for some $C$ and all $R>0$. Then, there exist some $R_0>0$, $e\in\Sph^{n-1}$, $N\in\N$, as well as smooth functions $f_i:\R^{n-1}\setminus B_{R_0}'\to\R$ for each $i=1,...,N$ with $f_1<...<f_N$, such that, up to a rotation,
    \begin{equation}
        \Sigma \setminus \left(B_{R_0}'\times[-R_0,R_0]\right)=\bigcup_{i=1}^N {\rm graph}\, f_i\,.
    \end{equation}
    Moreover, there are $b_i,c_i\in\R$ and $d_i\in \R^{n-1}$ for each $i=1,...,N$ such that
    $$
    f_i(y)=\begin{cases}
        b_i+c_i\log|y| +d_i\cdot\frac{y}{|y|^2} +O(|y|^{-1-\beta}), &\mbox{if}\ \  n=3,\\
        b_i+c_i\frac{1}{|y|^{n-3}} +d_i\cdot\frac{y}{|y|^{n-1}}+O(|y|^{2-n-\beta}), &\mbox{if}\ \  4\leq n\leq 7.
    \end{cases}
    $$
\end{theorem}
Some remarks are in order.
\begin{remark}\label{rmk:finindcond}
    The finite index condition is equivalent to stability away from some ball $B_M$. In fact, fixing $C$ and an outer stability radius $M$ in Theorem~\ref{thm:TyskIndex}, carefully tracing our proof would show that $N$, $R_0$ and the $b_i,c_i,d_i,y_i$ satisfy uniform bounds depending exclusively on $C,M$. Likewise, given $\beta\in(0,1)$, the $O(|y|^{2-n-\beta})$ error term has an implicit constant depending exclusively on $C,M,\beta$.
\end{remark}

    \begin{remark}
        It is worth highlighting the differences between our approach and the original one in \cite{Tysk89, Schoen83}. In both cases, the first step is to obtain a large-scale flatness condition (with respect to some directions $e_R$, possibly rotating as the scale $R$ goes to infinity) by a blow-down argument, using that---by the finite index condition---$\Sigma$ is stable away from some ball. From that point onward the arguments diverge:
        \begin{itemize}
            \item In \cite{Tysk89}, the the global graphicality and uniqueness of tangent cone for the large-scale components of $\Sigma$ are obtained directly, by using the finite index condition once again\footnote{Essentially, \cite{Tysk89} argues that if the normal vector oscillated between $e_1\neq e_2$, the Jacobi field $\nu_\Sigma\cdot (e_1-e_2)$ would have infinitely many compact nodal domains, contradicting finiteness of index.}. The asymptotics for the resulting graphs (which have then bounded slope) can then be obtained via Green kernel estimates as in \cite{Schoen83}.
            \item In our approach, which does not use the finiteness of index again, we apply the improvement-of-flatness result in Theorem~\ref{thm:improvminimalH} to the ends (i.e. large-scale components) of $\Sigma$. Letting $R\to\infty$, their graphicality and strong flatness follow. The refined asymptotics can be directly obtained from this.
        \end{itemize}
            Our PDE/regularity approach serves as a robust and flexible replacement for the argument in \cite{Tysk89}, suitable for a broad class of variational problems, as well as for dealing with general critical points or local settings.
    \end{remark}
\begin{remark}
    The area growth condition has long been expected to be redundant, and essentially equivalent to the resolution of the stable Bernstein problem on the rigidity of (complete, smoothly immersed, two-sided) stable minimal hypersurfaces in $\R^n$ for $3\leq n \leq 7$. They have both been confirmed for $3\leq n \leq 6$ via the works \cite{FS80, FC85, CL24, CLMS25, Maz24}.
\end{remark}

\section{Improvement of flatness in annuli}
\subsection{Preliminaries}
Consider the following standard consequence of curvature estimates:
\begin{lemma}\label{lem:rjbagbda}
    Let $\Sigma$ be an embedded smooth hypersurface in $B_1\subset \R^n$, with $(\overline{\Sigma}\setminus \Sigma)\cap \overline{B_1}\subset \partial B_1$, and assume that  $|\mathrm{II}_\Sigma(p)|\leq C_0$ for all $p\in \Sigma\cap B_1$, for some $C_0> 0$. There are $\delta=\delta(C_0,n)\in(0,1/2)$ and $C=C(n,C_0)$ such that the following holds:
    
    For every $p\in \Sigma\cap B_{1/2}$, up to a rotation mapping the normal to $\Sigma$ at $p$ to $e_n$ and a translation mapping $p$ to $0$, the following holds. Let $\Gamma$ denote the connected component of \[\Sigma\cap \{x=(x',x_n) \in B_{1} \ :  |x'|<\delta\}\]
    containing $p$. Then,
    \begin{equation*}
        \Gamma={\rm graph} \, g, \quad \mbox{where}\quad  g(0)=0,\quad|\nabla g|\leq \tfrac{1}{8}, \quad\mbox{and} \quad |D^2 g|\leq C \,.
    \end{equation*}
\end{lemma}
\begin{proof}
This is the standard uniform graph lemma; see \cite[(4.4)]{KorevaarKusnerSolomon1989} or   \cite[Lemma 12.4]{PerezNotesMinimalCMCSurfaces}.
\end{proof}

We will need the following refinement for hypersurfaces which are contained in a thin slab within an annulus:
 
\begin{lemma}\label{lem:mf3rt7qwg0tba}
    Let $\Sigma$ be an embedded smooth hypersurface in $B_{20}\setminus \overline{B_{10}}$ with $(\overline{ \Sigma}\setminus \Sigma) \cap (B_{20}\setminus \overline{B_{10}})=\emptyset$ and
    $$|\mathrm{II}_{\Sigma}|\leq C_0\quad\mbox{in}\quad B_{20}\setminus \overline{B_{10}}.$$
   There are $\ep_0=\ep_0(C_0,n) >0$ and $C=C(C_0,n)$ such that: If 
   \begin{equation}\label{hwiohthiow1}
         \Sigma\cap (B_{20}\setminus \overline{B_{10}})\subset \{|x_n|\leq \ep\}
    \end{equation}
    for some $\ep \le \ep_0$, then there exist $N\in\N$ and smooth $f_i:B_{19}'\setminus \overline{B_{11}'}\to [-\ep,\ep]$ such that
    $$\Sigma \cap  \left((B_{19}'\setminus \overline{B_{11}'})\times[-5,5]\right)=\bigcup_{i=1}^N {\rm graph}\, f_i,$$
    with $ |D^2 f_i|\leq C$. In addition, $\Sigma\cap (B_{18}\setminus \overline{B_{12}})$ is connected precisely if $N=1$.
\end{lemma}
\begin{remark}
From \eqref{hwiohthiow1} and $19^2+5^2<20^2$ we directly get
\begin{equation}\label{eq:slab1}
    \Sigma \cap  \big( (B_{19}'\setminus \overline{B_{11}'}) \times [-5,5]\big)\subset \Sigma \cap (B_{20}\setminus \overline{B_{10}})\subset B_{20}'\times[-\ep,\ep].
\end{equation}
Furthermore, assuming $\ep_0 < 1$,
\begin{equation}\label{eq:slab2}
    \Sigma \cap (B_{18}\setminus \overline{B_{12}})\subset (B_{19}'\setminus \overline{B_{11}'})\times[-\ep,\ep].
\end{equation}
We check \eqref{eq:slab2}: If $x = (x', x_n) \in \Sigma\cap (B_{18}\setminus \overline{B_{12}})$ then $|x|\geq 12$, and  $|x_n|\leq \ep_0< 1$ by assumption. Hence, $|x'|=\sqrt{|x|^2 - |x_n|^2}\geq \sqrt{12^2-1^2}>11\,.$
\end{remark}
\begin{proof}[Proof of Lemma~\ref{lem:mf3rt7qwg0tba}]
By Lemma~\ref{lem:rjbagbda}---appropriately rescaled---and the bound on the second fundamental form, there is $\delta >0$ (depending only on $n$ and $C_0$) such that: for any $p\in \Sigma \cap  \big( (B_{19}'\setminus \overline{B_{11}'}) \times [-5,5]\big)$, the connected component of $\Sigma$ in a small cylinder (of size $\delta$), centered at $p$, is a flat Lipschitz graph (in some direction, a priori depending on $p$). 
By \eqref{eq:slab1}, taking $\ep_0$ small enough (depending on the previous $\delta$), we see that $e_n$ is also a direction of graphicality (although the Lipschitz bound may change to, say, $\tfrac{1}{2}$ instead of $\tfrac{1}{8}$).

Since $\Sigma \cap  \big( (B_{19}'\setminus \overline{B_{11}'}) \times [-5,5]\big)$ is covered by such flat Lipschitz graphs (in the $e_n$ direction) and the radius of the $(n-1)$-balls of where these graphs are defined is bounded by below by $\delta/2$, we see that $\Sigma \cap  \big( (B_{19}'\setminus \overline{B_{11}'}) \times [-5,5]\big)$ decomposes into a union of horizontal graphs over 
$B_{19}'\setminus \overline{B_{11}'}$. The embeddedness of $\Sigma$ and the closedness hypothesis $(\overline{ \Sigma}\setminus \Sigma) \cap (B_{20}\setminus \overline{B_{10}})=\emptyset$ show that the graphs are ordered, and that there are finitely many of them (as, otherwise, by closedness we would have a vertical accumulation point belonging to $\Sigma$, but then this would contradict embeddedness)

In other words, there exist $N\in \N$ and $f_i: B_{19}'\setminus \overline{B_{11}'}\to [-\ep,\ep]$ such that
    $$\Sigma\cap \left((B_{19}'\setminus \overline{B_{11}'})\times [-5,5] \right)=\bigcup_{i=1}^{N}{\rm graph}\,f_i\,$$

    Finally, the fact that $\Sigma\cap (B_{18}\setminus \overline{B_{12}})$ is connected precisely if $N=1$ follows then from \eqref{eq:slab2} plus elementary considerations, up to making $\ep_0$ (and thus $\delta$) small enough.
\end{proof}
We will also need the following ``compactness lemma'':
\begin{lemma}
\label{lem:compactnessminimal}
    Let $n \ge 3$ and $\alpha \in (0, 1)$. For any $\ep > 0$ there exists $\delta=\delta(n,\alpha,\ep)> 0$ such that the following holds. 
    
    Define $\psi(t): = \max(t^{4-n-\alpha}, t^{1+\alpha})$. Let $v\in C^{\infty}(B_{1/\delta}'\setminus \overline{B_\delta'})$ satisfy
    $$
    {\rm Tr}(A(y) D^2 v)=0 \quad \text{in}\quad B_{1/\delta}'\setminus \overline{B_\delta'},
    $$
    where 
    \[
    \sup_{y\in B_{1/\delta}'\setminus \overline{B_\delta'}}\|A(y) - {\rm Id}\|_2 \le \delta \qquad \text{and}\qquad |v(y)| \le \psi(|y|) \quad \text{for}\quad y\in B_{1/\delta}'\setminus \overline{B_\delta'}.
    \]
    Then, there are $a\in \R^{n-1}$ and $b\in \R$ such that 
    \[
    |v(y) - a\cdot y - b|\le \ep\qquad\text{in}\quad B_4'\setminus \overline{B_{1/4}'},
    \]
    where $|a|+|b|\le C$ for some $C=C(n)$. 
\end{lemma}
\begin{proof}
    Assume for contradiction that the statement does not hold. Then, there are $\ep_\circ > 0$, and a sequence $v_k$ as in the statement with $\delta = 1/k$ and $k\to\infty$, such that the thesis fails for $\ep = \ep_\circ$. 

    By standard Cordes--Nirenberg estimates \cite{Cordes1956, Nirenberg1954}, the $v_k$ satisfy $C_{loc}^{1,\alpha}$ bounds in $\R^{n-1}\setminus\{0\}$. By Arzel\`a-Ascoli, up to passing to a subsequence $v_k$ converges locally in $C^1$ to some $v_\infty$, which is then a viscosity (thus strong) solution to
    \[
    \Delta v_\infty=0\quad\mbox{in}\quad\R^{n-1}\setminus\{0\}\,.
    \]
    Additionally,
    \[
    |v_\infty(y)| \le \psi(|y|)=\max(|y|^{4-n-\alpha}, |y|^{1+\alpha}).
    \]
    In particular, this implies that $v_\infty=o(|y|^{2-(n-1)})$ as $|y|\to 0$, so that the singularity at the origin is removable by a standard barrier argument. Then, using that $v_\infty=o(|y|^2)$ as $|y|\to\infty$, the usual Liouville theorem for entire harmonic functions gives that $v_\infty$ is affine, yielding a contradiction for $k$ large enough.
\end{proof}
\subsection{Proof of Theorem~\ref{thm:improvminimalH}}
The proof of Theorem~\ref{thm:improvminimalH} will make use of the following ``iteration step'':
\begin{lemma}
\label{lem:Hu1minimal}
    Let  $n\ge3$, $\alpha\in (0,1)$, and $C_0,\ep>0$. There exist $\delta=\delta(n,\alpha,C_0,\ep)>0$ and $\eta_0=\eta_0(n,\alpha,C_0,\ep)>0$ such that the following holds.

    Let $\Sigma$ be a minimal hypersurface 
    in $D = B_{1/\delta}\setminus \overline{B_\delta} \subset \R^n$, satisfying
    \[
      |{\rm II}_\Sigma|\le C_0\quad  \mbox{in}\quad  D,
  \]
  and assume that $\Sigma\cap(B_{2r}\setminus \overline{B_{r/2}})$  is connected for every $r\in(2\delta,1/(2\delta))$.
  
    Set $\psi(t): = \max(t^{4-n-\alpha}, t^{1+\alpha})$. Assume that there are $|b|<\eta_0$, $e\in\Sph^{n-1}$, and $0< \eta\le \eta_0$, such that
    \begin{equation}\label{growth00minimal}
        H_b(\Sigma, e, t) \le \eta\, \psi(t) \qquad \text{for all} \quad  2\delta <t<\tfrac 1{2\delta}.
    \end{equation}
      
    Then, we have
    \[
        H(\Sigma,1) \le \varepsilon  \eta.
    \]
\end{lemma}
\begin{proof}
After a rotation, we can assume that $e=e_n$. Moreover, after a translation we can also assume that $b=0$, up to showing then that
$$H(\Sigma,r) \le \varepsilon  \eta\quad\mbox{for every}\quad r\in (1/2,2)$$
instead\footnote{Indeed, we can ensure that $|b|< 1/4$ by making $\eta_0$ small; moreover, given $x\in B_{1/4}$, we have $B_2\setminus \overline{B_{1/2}}\subset \bigcup_{r\in(1/2,2)}B_{2r}(x)\setminus \overline{B_{r/2}(x)}$.}.

\noindent {\bf Step 1.} Graphicality.

Let $\delta\in(0,1)$ to be chosen later. Set $\bar \ep:=\max_{2\delta<t<\frac{1}{2\delta}} \eta\psi(t)$, so that
$$H_0(\Sigma, e_n, t) \le \eta\psi(t)\leq \bar\ep \qquad \text{for all} \quad  2\delta <t<\tfrac 1{2\delta}.$$
Up to making $\eta_0$ small enough (depending only on $n,\alpha,C_0,\delta$), which in turn makes $\bar\ep$ as small as desired, we can apply Lemma \ref{lem:mf3rt7qwg0tba} and equation \eqref{eq:slab2} (appropriately rescaled\footnote{It suffices to consider $\widetilde \Sigma_r:=\frac{10}{r}\left(\Sigma\cap (B_{2r}\setminus \overline{B_{r}})\right)$, which satisfies $|\mathrm{II}_{\widetilde\Sigma_r\cap (B_{20}\setminus \overline{B_{10}})}|\leq \frac{r}{10}C_0\leq \frac{1}{20\delta}{C_0}$ and $\widetilde \Sigma_r\cap (B_{20}\setminus \overline{B_{10}})\subset\{|x_n|\leq \frac{10}{r}\bar\ep\}\subset\{|x_n|\leq \frac{10}{\delta}\bar\ep\}$.}, recall \eqref{eq:scalingminimal}) to $\Sigma\cap (B_{2r}\setminus \overline{B_{r}})$, for every $r\in[\delta,\frac{1}{2\delta}]$. Piecing this information together, this gives graphicality in the $x_n$-direction: We find that
$$\Sigma \cap (B_{1/(4\delta)}\setminus \overline{B_{4\delta}})\subset (B_{1/(2\delta)}'\setminus \overline{B_{2\delta}'})\times[-\bar\ep,\bar\ep]\,,$$
    and there is some smooth $f:B_{1/(2\delta)}'\setminus \overline{B_{2\delta}'}\to [-\bar\ep,\bar\ep]$ such that
    $$ \Sigma \cap  \left((B_{1/(2\delta)}'\setminus \overline{B_{2\delta}'})\times[-\bar\ep,\bar\ep]\right)= {\rm graph}\, f\,.$$
    Moreover, we get that $|D^2 f|\leq C$, where $C=C(C_0,n,\delta)$.

Let $A(y)$ denote the $(n-1)\times (n-1)$ matrix with components
$$(A(y))_{ij}:=\delta_{ij}
-\frac{\partial_i f(y)\,\partial_j f(y)}
{1+|\nabla f|^2(y)},$$
chosen so that the minimal graph equation in nondivergence form becomes
\begin{equation}
    \label{eq:TrA_nondiv}
    {\rm Tr}(A(y)D^2 f)=0.
\end{equation}
Now, a simple interpolation ensures that $|\nabla f|$ can be made as small as wanted, up to making $\eta_0$ even smaller. In particular, if $\eta_0$ is small enough (depending, as before, only on $n,\alpha,C_0,\delta$), we can make sure that  
\[
    \sup_{y\in B_{1/(4\delta)}'\setminus \overline{B_{4\delta}'}}\|A(y) - {\rm Id}\|_2 \le \delta.
\]

\noindent {\bf Step 2.} Vertical rescaling and conclusion.

Consider $v:=\frac{f}{\eta}$, so that
$$|v(y)|\leq \psi(|y|)\quad\mbox{and}\quad{\rm Tr}(A(y)D^2 v)=0\quad\mbox{for}\quad y\in B_{1/(4\delta)}'\setminus \overline{B_{4\delta}'}\subset\R^{n-1}.$$
Fixing $\delta$ small depending on $n,\alpha,C_0,\ep$, which in turn fixes $\eta_0$ in terms of $n,\alpha,C_0,\ep$ too, we can then apply Lemma~\ref{lem:compactnessminimal} (with $4\delta$ in place of $\delta$) to $v$. This shows that
\[
    |v(y) - a\cdot y - b|\le \frac{\ep}{2}\qquad\text{in}\quad B_4'\setminus \overline{B_{1/4}'},
    \]
    or
    \[
    |f(y) - \widetilde a\cdot y - \widetilde b|\le \frac{\ep}{2}\eta \qquad\text{in}\quad B_4'\setminus \overline{B_{1/4}'},
    \]
    for some $\widetilde a\in\R^{n-1}$ and $\widetilde b\in\R$ with $|\widetilde a|+|\widetilde b|\le C\eta$.
Letting
$$e:=\frac{(-\widetilde a,1)}{\sqrt{1+|\widetilde a|^2}}\qquad\mbox{and}\qquad b:=\frac{-\widetilde b}{\sqrt{1+|\widetilde a|^2}},$$
which are chosen so that
$${\rm graph}(\widetilde a\cdot y+\widetilde b)=\{e\cdot x+b=0\},$$
we immediately see that $H_b(\Sigma,e,1)\leq \ep\eta$---up to possibly making $\delta$ and $\eta_0$ even smaller one last time.
\end{proof}

We are now ready for the proof of Theorem \ref{thm:improvminimalH}:
\begin{proof}[Proof of Theorem \ref{thm:improvminimalH}]
We argue by contradiction. Let $C_\star$ be a positive constant, to be chosen large enough (depending only on $n$ and $\alpha$). Put $\phi(t): = t(t^{3-n-\alpha} +(t/R)^{\alpha})$. Define
\[
Q_\star :=  \sup_{r\in [2,R/2]}  \frac{H(\Sigma,r)}{\phi(r)},
\]
and assume for contradiction that  
\[
Q_\star \ge C_\star \eta.
\]

By continuity, there exist $r_\star\in[2,R/2]$ and $b_\star\in \R$, $e_\star\in \mathbb{S}^{n-1}$ such that
\[
H(\Sigma,r_\star) = Q_\star \phi(r_\star) \qquad\mbox{and}\qquad H(\Sigma,r_\star)= H_{b_\star}(\Sigma,e_\star,r_\star).
\]
More generally, for every $r\in [2,R/2]$ there exist $b_r\in\R$ and $e_r\in\Sph^{n-1}$  such that
\begin{equation}\label{89grqbgpargh}
    H(\Sigma,r) \le  Q_\star \phi(r)=H(\Sigma,r_\star)\frac{\phi(r)}{\phi(r_\star)} \qquad\mbox{and}\qquad H(\Sigma,r)=H_{b_r}(\Sigma,e_r, r)\le\eta r,
\end{equation}
where we have additionally used assumption \eqref{flatnesshypothesisthmminimal} in the last inequality.

We proceed in several steps.

\noindent {\bf Step 1.} Rescaling.

Consider $\bar \Sigma := \frac{1}{r_*}\Sigma$. Denoting
$$t = r/r_\star\in (2/r_\star, R/(2r_\star)),\quad \bar b_t = b_r/r_\star,\quad\mbox{and}\quad\bar e_t = e_r,$$ \eqref{89grqbgpargh} becomes (recall \eqref{eq:scalingminimal})
\begin{equation}\label{growthctrl12minimal}
 H_{\bar b_t}(\bar \Sigma, \bar e_t, t) \le  \frac{H(\Sigma,r_\star)}{r_\star}\frac{\phi(t r_\star)}{\phi(r_\star)} \qquad \mbox{and}\qquad H_{\bar b_t}(\bar \Sigma, \bar e_t, t)\leq\eta t .
\end{equation}

Noticing that 
\begin{equation}\label{doublingminimal}
\frac{\phi(t r_\star)}{\phi(r_\star)} \le \max(t^{4-n-\alpha}, t^{1+\alpha}) =: \psi(t),
\end{equation}
this gives
\begin{equation}\label{doubling2minimal}
H_{\bar b_t}(\bar \Sigma, \bar e_t, t) \le  \eta_\star \psi(t),\qquad\mbox{with}\quad \eta_\star:= \frac{H(\Sigma,r_\star)}{r_\star} \le \eta \le  \eta_0.
\end{equation}
We want to apply Lemma~\ref{lem:Hu1minimal} to $\bar\Sigma$. Note that \eqref{flatnesshypothesisthmminimal} implies
\begin{equation}\label{boundbrminimal}
|b_r|< CH(\Sigma,r), \quad \mbox{thus} \quad |\bar b_t|< C\frac{H(\Sigma, tr_\star)}{r_\star} \le C\eta t.
\end{equation}
\noindent {\bf Step 2.} Coefficient comparison.

By definition of $H_{\bar b_t}(\bar \Sigma, \bar e_t, t)$ we get
\begin{equation}\label{barsigmainclusionminimal}
    \bar\Sigma\cap (B_{2t}\setminus \overline{B_{t/2}})\subset \left\{|\bar e_t\cdot x-\bar b_t|\leq H_{\bar b_t}(\bar \Sigma, \bar e_t, t)\right\}.
\end{equation}
Applying this also with $2t$ in place of $t$ and combining both results,
\begin{equation}\label{barsigmainclusion2minimal}
    \bar\Sigma\cap (B_{2t}\setminus  \overline{B_{t}})\subset \left\{|\bar e_t\cdot x-\bar b_t|\leq H_{\bar b_t}(\bar \Sigma, \bar e_t, t)\right\}\cap \left\{|\bar e_{2t}\cdot x-\bar b_{2t}|\leq H_{\bar b_{2t}}(\bar \Sigma, \bar e_{2t}, 2t)\right\},
\end{equation}
and it is then not hard to see\footnote{To be precise, since $H_{\bar b_t}(\bar \Sigma, \bar e_t, t)\leq\eta t$ by \eqref{growthctrl12minimal}, up to making $\eta_0$ small enough (recall that $\eta\leq \eta_0$), by Lemma \ref{lem:mf3rt7qwg0tba} and equations \eqref{eq:slab1}--\eqref{eq:slab2} (appropriately rescaled, consider $\widetilde \Sigma_t:= \frac{10}{t}\Sigma$) we can write a piece of $\bar \Sigma\cap (B_{2t}\setminus \overline{B_{t}})$ as a graph in the direction $\bar e_t$. Rewriting \eqref{barsigmainclusion2minimal} for this graph instead, and applying the triangle inequality, the comparison immediately follows.} that
\[
\begin{split}
t|\bar e_{2t} -\bar e_{t}| + |\bar b_{2t}-\bar b_{t}| \le C (H_{\bar b_t}(\bar \Sigma, \bar e_t, t)+ H_{\bar b_{2t}}(\bar \Sigma, \bar e_{2t}, 2t)).
\end{split}
\]

By \eqref{doubling2minimal} we then have
\[
\begin{split}
t|\bar e_{2t} -\bar e_{t}| + |\bar b_{2t}-\bar b_{t}|\le C\eta_\star (\psi(2t)+ \psi(t)),
\end{split}
\]
which summing the geometric series yields (recall $e_\star = \bar e_1$), for $2/r_\star <t<R/(2r_\star)$,
\[
t|\bar e_{t}-  e_\star|+ |\bar b_{t}- \bar b_1| \le   C\eta_\star \psi(t).
\]
Combined with \eqref{doubling2minimal}, this implies (recall Definition~\ref{defi:heightminimal}) that
\[
H_{\bar b_1}(\bar \Sigma, \bar e_1, t) \le \bar \eta_\star\psi(t)\qquad \text{for all}\quad t\in (2/r_\star,R/(2r_\star)), \qquad\mbox{with}\quad \bar \eta_\star:= C\eta_\star\le C \eta_0.
\]
As in the rest of the proof, here $C = C(n, \alpha,C_0)$. Moreover, $|\bar b_1| \le C\eta\le C\eta_0$ by \eqref{boundbrminimal}.

\noindent {\bf Step 3.} Estimate for $r_*$.

By \eqref{flatnesshypothesisthmminimal} we know that $H(\Sigma,r)/r\le \eta$ for all $r\in (2,R/2)$, thus
\[
Q_\star \le \frac{\eta}{r_\star^{3-n-\alpha} + (r_\star/R)^{\alpha}}\le \eta\max\left\{\frac{1}{r_\star^{3-n-\alpha}},\frac{1}{(r_\star/R)^{\alpha}}\right\}.
\]
Applying this and $Q_*\geq C_*\eta$ twice, we find that
\begin{equation}
    \label{eq:boundonrstarminimal}
C_\star^{\frac{1}{n+\alpha-3}}\le (Q_\star/\eta)^{\frac{1}{n+\alpha-3}}\le r_\star \le R (\eta/Q_\star)^{\frac{1}{\alpha}} \le R C_\star^{-\frac{1}{\alpha}} ,
\end{equation}
so that $r_*\in [C_\star^{\frac{1}{n+\alpha-3}}, R C_\star^{-\frac{1}{\alpha}}]$. Since our arguments above were valid for all $t\in (2/r_\star,R/(2r_\star))$, given $\delta>0$ this includes all $t\in (\delta,\frac{1}{\delta})$ up to making $C_*$ large enough.

\noindent {\bf Step 4.} Conclusion.

Let $\ep>0$, to be fixed. Combining all of the above, by choosing $C_\star$ large and $\eta_0$ small (depending only on $n$,  $\alpha$, $C_0$ and $\ep$), the setting of Lemma~\ref{lem:Hu1minimal} is satisfied. This gives
\[
H(\bar \Sigma, 1) \le \ep\bar \eta_\star =C\ep\frac{H(\Sigma,r_\star)}{r_\star} ,
\]
or (since $\bar \Sigma=\frac{1}{r_\star}\Sigma$, and thus $H(\bar \Sigma, 1)=\frac{H(\Sigma,r_\star)}{r_\star}$), equivalently,
$$H(\Sigma,r_\star) \le C \ep H(\Sigma,r_\star) .$$
Once again, we emphasise that $C=C(n,\alpha, C_0)$. Fixing finally $\ep$ so that $C\ep =\frac{1}{2}$, which in turn fixes the choice of $C_\star$ and $\eta_0$ in terms of $n$, $\alpha$ and $C_0$, we arrive at a contradiction. 
\end{proof}

\section{Minimal hypersurfaces with finite Morse index}

An interesting case is obtained by letting $R\to\infty$ in Theorem~\ref{thm:improvminimalH}.  
\begin{corollary}\label{cor:Rtoinftyminimal}
    Let $n\geq 3$. Let $C_0>0$ and $\alpha\in(0,1)$. There are $C=C(n,\alpha,C_0)$ and $\eta_0=\eta_0(n,\alpha,C_0)>0$ such that the following holds:

Let $\Sigma$ be a complete, embedded minimal hypersurface in $\R^n\setminus B_1$.
Let $0<\eta\leq\eta_0$, and assume that for every $r\in [2,\infty)$ we have that $\Sigma\cap (B_{2r}\setminus \overline{B_{r/2}})$ is connected, $|{\rm II}_{\Sigma\cap (B_{2r}\setminus \overline{B_{r/2}})}| \le C_0/r$, and
\begin{equation}\label{flatnesshypothesisthmminimal2}
    H_0(\Sigma,r) \le \eta r.
\end{equation}
Then, $\Sigma$ is graphical in $\R^n\setminus B_4$. More precisely, up to a rotation,
$$\Sigma \cap (\R^n\setminus \overline{B_4})\subset (\R^{n-1}\setminus \overline{B_2'})\times \R,$$
and there is a smooth $f:\R^{n-1}\setminus B_2'\to\R$ such that 
        $$\Sigma\cap \left( (\R^{n-1}\setminus \overline{B_2'}) \times \R\right)= {\rm graph}\, f\,.$$
    Moreover, there is some $b\in\R$ such that
        \begin{equation}\label{assimptotic}
            |f(y)-b|\leq C\frac{\eta}{|y|^{n-4+\alpha}}\qquad\text{for}\quad y\in \R^{n-1}\setminus B_2'. 
        \end{equation}
    \end{corollary}
    \begin{proof}
    We assume that $n>3$; the case $n=3$ follows with an identical proof up to putting $b=0$.
    
    Let $\gamma:=4-n-\alpha<0$. By Theorem~\ref{thm:improvminimalH} applied with $R\to\infty$, we know that
    $$
    H(\Sigma,r)\leq C\eta r^{\gamma}\qquad\mbox{for all}\quad r\in(2,\infty).
    $$
    By definition, there are then $e_k\in\Sph^{n-1}$ and $b_k\in\R$ for every $k\geq 2$ such that
\begin{align}\label{eq:37y4tgweeub}
    \Sigma \cap (B_{2^{k+1}}\setminus \overline{B_{2^{k-1}}})&\subset \{|e_k\cdot x-b_k|\leq C\eta 2^{\gamma k}\}\,.
\end{align}

Applying \eqref{eq:37y4tgweeub} with $k$ and $k+1$, we deduce in particular that
\begin{align}\label{eq:t2831t41gtq}
    \Sigma \cap (B_{2^{k+1}}\setminus \overline{B_{2^{k}}})&\subset \{| e_k\cdot x-b_k|\leq C\eta 2^{\gamma k}\}\cap \{| e_{k+1}\cdot x-b_{k+1}|\leq C\eta 2^{\gamma (k+1)}\}\,.
\end{align}
Then,  the local graphicality (coming from Lemma~\ref{lem:mf3rt7qwg0tba}) of $\Sigma$ and the triangle inequality  give that
$$2^k|e_{k+1}-e_k|+|b_{k+1}-b_k|\leq C\eta 2^{\gamma k}\,.$$
From this, we find that
$$\sum_{i\geq k} |e_{i+1}-e_i|\leq \sum_{i\geq k} C\eta2^{(\gamma-1)i} \leq C_\alpha\eta 2^{(\gamma-1) k}\qquad\mbox{and}\qquad \sum_{i\geq k} |b_{i+1}-b_i| \leq C_\alpha\eta 2^{\gamma k},$$
and in particular the sequences $e_k$ and $b_k$ are both Cauchy. This shows the existence of some $e_\infty\in \Sph^{n-1}$ and $b_\infty\in\R$ with
\begin{equation*}
    |e_\infty-e_{k}|\leq \sum_{i\geq k} |e_{i+1}-e_i|\leq C_\alpha\eta 2^{(\gamma-1)k}\quad\mbox{and}\quad |b_\infty-b_{k}|\leq \sum_{i\geq k} |b_{i+1}-b_i|\leq C_\alpha\eta 2^{\gamma k}\,,
\end{equation*}
thus
\begin{equation}\label{eq:q7toagvg2}
    \Sigma \cap (B_{2^{k+1}}\setminus \overline{B_{2^{k-1}}})\subset \{|e_\infty\cdot x-b_\infty|\leq C_\alpha\eta 2^{\gamma k}\}\,\quad \forall k \ge 2. 
\end{equation}

Up to making $\eta_0>0$ smaller (recall that $\eta\leq\eta_0$), we can now just apply Lemma~\ref{lem:mf3rt7qwg0tba}---appropriately rescaled---to $\Sigma\cap B_{2^k}\setminus \overline{B_{2^{k-1}}}$ for every $k\geq 2$. This shows that, up to a rotation (so that $e_\infty = e_n$), 
$$\Sigma \cap (\R^n\setminus \overline{B_4})\subset (\R^{n-1}\setminus \overline{B_2'})\times \R$$
and moreover $\Sigma\setminus \left(\overline{B_{2}'}\times[-2,2]\right)$ is the graph of some $f:\R^{n-1}\setminus \overline{B_2'}\to\R$. Additionally, \eqref{eq:q7toagvg2} gives that
\begin{align}\label{eq:89ygoiubjbad}
    \Sigma \cap (B_{2^{k+1}}\setminus \overline{B_{2^{k-1}}})&\subset \{|x_n-b_\infty|\leq C_\alpha \eta 2^{\gamma k}\}\,,\quad\mbox{thus}\quad |f(x')-b_\infty|\leq CC_\alpha |y|^\gamma,
\end{align}
as we wanted.
\end{proof}
    This expansion can actually be a posteriori  improved to a higher order one:
    
    \begin{lemma}\label{lem:fundsolterm}
        In the setting of Corollary~\ref{cor:Rtoinftyminimal}, there are additionally some $d\in \R^{n-1}$ and $c\in\R$ such that
        $$\left|f-b-\frac{c}{|y|^{n-3}} - \frac{d\cdot y}{|y|^{n-1}}\right|\leq C\frac{\eta}{|y|^{n-2+\alpha}}\,,$$
            where $c|y|^{3-n}$ is replaced by $c\log |y|$ if $n=3$.
    \end{lemma}

\begin{proof}
By the asymptotics from Corollary~\ref{cor:Rtoinftyminimal} (see \eqref{assimptotic})
and standard interior estimates for minimal graphs we obtain, for $|y|>M$, 
\begin{equation}\label{eq:C2decay}
|f(y)-b|+|y|\,|\nabla f(y)|+|y|^2|D^2f(y)|
\le C\eta\,|y|^{4-n-\alpha}.
\end{equation}
Set $u:=f-b$. As in \eqref{eq:TrA_nondiv} we have
\[
\sum_{i, j = 1}^{n-1} \Big(\delta_{ij}-\frac{u_i u_j}{1+|\nabla u|^2}\Big)u_{ij}=0,
\qquad\text{hence}\qquad
\Delta u=\sum_{i, j = 1}^{n-1} \frac{u_i u_j}{1+|\nabla u|^2}\,u_{ij}.
\]
Therefore, also using   \eqref{eq:C2decay} we have
\[
|\Delta u(y)|\le C|\nabla u(y)|^2|D^2u(y)|\le C\eta^3 |y|^{8-3n-3\alpha}\le C\eta |y|^{8-3n-3\alpha}\qquad\text{for}\quad |y| > R. 
\]
Fix $\bar\alpha\in(0,1)$ and choose $\alpha\in(0,1)$ so that
$8-3n-3\alpha\le -(n+\bar\alpha)$ (for $n=3$ it suffices e.g.\ $\alpha\ge (2+\bar\alpha)/3$).
Then 
\begin{equation}\label{eq:lapdecay}
|\Delta u(y)|\le C\eta\,|y|^{-(n+\bar\alpha)}\qquad\text{for}\quad |y|>M.
\end{equation}

Now consider the Kelvin transform on $0<|x|<1/M$ with respect to $B_1'\subset \R^{n-1}$,
\[
v(x):=|x|^{2-(n-1)}\,u \Big(\frac{x}{|x|^2}\Big)
\qquad\text{so that}\qquad \Delta v(x)=|x|^{-n-1}\,\Delta u\Big(\frac{x}{|x|^2}\Big).
\]
With $y=x/|x|^2$ (so $|y|=|x|^{-1}$), \eqref{eq:lapdecay} yields
\begin{equation}\label{estlaplacian}
    |\Delta v(x)|\le C\eta\,|x|^{\bar\alpha-1}\qquad\text{for}\quad x\in B'_{1/M}.
\end{equation}

Moreover \eqref{eq:C2decay} implies $|v(x)|\le C\eta |x|^{\alpha-1}$ for $x\in B'_{1/M}$.

Assume now $n-1\ge 3$. Standard Newton-potential estimates\footnote{Write $v = \Gamma_{n-1} * (\Delta v {\mathbbm {1}}_{B'_{M}\setminus\{0\}}) + h$ where $\Gamma_{n-1}$ is the fundamental solution of the Laplacian in $\R^{n-1}$ and $h$ is a harmonic function. (We use  $|v(x)|\le C\eta |x|^{\alpha-1}$ to show that the singularity of $h$ at $0$ is removable if $n\ge4$, or removable after subtracting a suitable multiple of $\log$ for $n-1=2$.)}
  imply that $v$ extends across $0$ with $v\in C^{1,\bar\alpha}$
and, in particular,
\[
|v(x)-v(0)-\nabla v(0)\cdot x|\le C\eta\,|x|^{1+\bar\alpha}\qquad\text{for}\quad x\in B'_{1/M}.
\]
Set $c:=v(0)$ and $d_i:=-\partial_i v(0)$. Writing $x=y/|y|^2$ and undoing the Kelvin
transform gives, for $|y|$ large,
\[
u(y)=|y|^{2-(n-1)}\,v\!\Big(\frac{y}{|y|^2}\Big)
=\frac{c}{|y|^{n-3}}-\frac{d\cdot y}{|y|^{n-1}}
+O\!\Big(\frac{\eta}{|y|^{n-2+\bar\alpha}}\Big).
\]
Equivalently,
\[
\left|f-b-\frac{c}{|y|^{n-3}} +\frac{d\cdot y}{|y|^{n-1}}\right|
\le C\frac{\eta}{|y|^{n-2+\bar\alpha}}.
\]

A similar argument works  $n=3$, up to subtracting the right  multiple of the logarithmic fundamental solution.
\end{proof}

    To obtain Theorem~\ref{thm:TyskIndex}, we need a final lemma:
    \begin{lemma}\label{lem:TyskHypotheses}
        Let $3\leq n \leq 7$. Let $\Sigma\subset\R^n$ be a complete, embedded minimal hypersurface with finite Morse index, satisfying ${\rm Area}(\Sigma\cap B_R)\leq CR^{n-1}$ for some $C$ and all $R>0$.

        Then, there are some $R_0,C_0>0$ and $N\in\N$ such that the following hold:
        \begin{itemize}
            \item Let $\eta_0>0$. There is some $R_1=R_1(\Sigma,\eta_0)\geq 8R_0$ such that,  for every $r>R_1$, $H_0(\Sigma,r)\leq \eta_0 r$.
            \item $\Sigma\setminus \overline{B_{R_0}}=\bigcup_{i=1}^N \Sigma_i$, where the $\Sigma_i$ are embedded minimal hypersurfaces with $\overline \Sigma_i\setminus \Sigma_i\subset \partial B_{R_0}$. 
            \item For each $1\le i \le N$, $\Sigma_i\cap (B_{2r}\setminus \overline{B_{r/2}})$ is connected for every $r>2R_0$, and $|\mathrm{II}_{\Sigma_i\cap (B_{2r}\setminus \overline{B_{r/2}})}|\leq C_0/r$.
        \end{itemize}
    \end{lemma}
    \begin{remark}\label{rmk:finindcond2}
    Similarly to the discussion in Remark~\ref{rmk:finindcond} in Theorem~\ref{thm:TyskIndex}, the finite index condition is only used to obtain stability away from some ball $B_M$, and fixing $C$ and $M$ in Lemma~\ref{lem:TyskHypotheses} would give a uniform bound on $R_0,C_0,N$.
\end{remark}
    \begin{proof}
        The finite index condition implies that $\Sigma$ is stable outside of some large ball $B_{R_0}$; we will make $R_0$ larger in what follows when needed. By the area bounds and the restriction $3\leq n \leq 7$, the estimates in \cite{SS81} imply that $\Sigma$ satisfies curvature estimates $|\mathrm{II}_\Sigma|\leq \frac{C_0}{|x|}$ for some uniform constant $C_0$.
        
        We start by showing the first bullet. Assume, for contradiction, that there is $\eta_0>0$ and a subsequence $r_k\to\infty$ with $H_0(\Sigma,r_k)> \eta_0 r_k$. Consider $\widetilde \Sigma_k:=\frac{1}{r_k}\Sigma$, which satisfy then $H_0(\widetilde \Sigma_k,1)> \eta_0$. Now, the area and curvature estimates (together with standard bootstrap estimates for minimal hypersurfaces and Arzel\`a--Ascoli theorem) imply that, up to passing to a subsequence, the sequence $\{\widetilde \Sigma_k\}_k$ converges in $C^2_{loc}(\R^n\setminus \{0\})$ to a regular (away from the origin), stable minimal hypersurface $\mathcal C$. Moreover, the monotonicity formula implies that $\mathcal C$ is actually a cone. But then, Simons' classification \cite{Simons68} implies that $\mathcal C$ is a hyperplane, thanks to the restriction $3\leq n \leq 7$ once again, giving a contradiction with $H_0(\widetilde \Sigma_k,1)> \eta_0$ for $k$ large enough.
        
        Now, let $\Sigma_i$ denote the components of $\Sigma$ away from a large ball $B_{R_0}$. With the first bullet at hand and the curvature estimates, arguing exactly as in Lemma~\ref{lem:mf3rt7qwg0tba} with each component $\Sigma_i$ (see \cite[Lemma 3.4]{Flo25} for full details),  gives the two last bullets and concludes the proof.
    \end{proof}
    Then, we can give:
    \begin{proof}[Proof of Theorem~\ref{thm:TyskIndex}]
    It suffices to consider the $\Sigma_i$ given by Lemma~\ref{lem:TyskHypotheses} and to apply Corollary~\ref{cor:Rtoinftyminimal} and Lemma~\ref{lem:fundsolterm} (appropriately rescaled) to them\footnote{To be precise, Corollary~\ref{cor:Rtoinftyminimal} gives some $e_i\in\Sph^{n-1}$, possibly depending on $i$, for which its thesis holds for $\Sigma_i$. To obtain Theorem~\ref{thm:TyskIndex}, we need instead a single $e\in\Sph^{n-1}$ which works simultaneously for all of the $\Sigma_i$. However, embeddedness immediately forces all of the $e_i\in\Sph^{n-1}$ to coincide, as otherwise the $\Sigma_i$ would intersect each other along an $(n-2)$-dimensional submanifold.}.
    \end{proof} 


\printbibliography

\end{document}